\documentclass[11pt]{amsart}
\usepackage{amsmath, amssymb}
\usepackage{amsfonts}
\usepackage{mathrsfs}
\usepackage[arrow,matrix,curve,cmtip,ps]{xy}

\usepackage{epsfig}

\usepackage{amsthm}

\allowdisplaybreaks

\newtheorem{theorem}{Theorem}[section]
\newtheorem{lemma}[theorem]{Lemma}
\newtheorem{proposition}[theorem]{Proposition}

\newtheorem*{theorem*}{Theorem}
\theoremstyle{remark}

\newtheorem{definition}[theorem]{Definition}

\numberwithin{equation}{section}

\newcommand{\N}{\mathbb{N}}

\newcommand{\C}{\mathbb{C}}

\newcommand{\im}{\operatorname{im }}

\newcommand{\Ad}{\operatorname{Ad}}

\newcommand{\clspan}{\operatorname{\overline{\textnormal{span}}}}

\newcommand{\Hi}{\mathcal{H}}

\newcommand{\gae}{\lower 2pt \hbox{$\, \buildrel {\scriptstyle >}\over {\scriptstyle
\sim}\,$}}

\newcommand{\lae}{\lower 2pt \hbox{$\, \buildrel {\scriptstyle <}\over {\scriptstyle
\sim}\,$}}

\newcommand{\MU}[1]{
\setbox0\hbox{$#1$}
\setbox1\hbox{$W$}
\ifdim\wd0>\wd1 #1^{\sim} \else \widetilde{#1} \fi
}

\begin{document}
\title{Naimark's problem for graph $C^*$-algebras}

\author{Nishant Suri and Mark Tomforde 
}

\address{Department of Mathematics \\ University of Houston \\ Houston, TX 77204-3008 \\USA}
\email{nsuri@math.uh.edu}
\email{tomforde@math.uh.edu}

\thanks{This work was supported by a grant from the Simons Foundation (\#527708 to Mark Tomforde)}

\date{\today}



\begin{abstract}

Naimark's problem asks whether a $C^*$-algebra that has only one irreducible $*$-representation up to unitary equivalence is isomorphic to the $C^*$-algebra of compact operators on some (not necessarily separable) Hilbert space.  This problem has been solved in special cases, including separable $C^*$-algebras and Type~I $C^*$-algebras.  However,  in 2004 Akemann and Weaver used the diamond principle to construct a $C^*$-algebra with $\aleph_1$  generators that is a counterexample to Naimark's Problem.  More precisely, they showed that the statement ``There exists a counterexample to Naimark's Problem that is generated by $\aleph_1$  elements." is independent of the axioms of ZFC.  Whether Naimark's problem itself is independent of ZFC remains unknown.  In this paper we examine Naimark's problem in the setting of graph $C^*$-algebras, and show that it has an affirmative answer for (not necessarily separable) AF graph $C^*$-algebras as well as for $C^*$-algebras of graphs in which each vertex emits a countable number of edges.

\end{abstract}

\maketitle

\section{Introduction}

In 1948 Naimark proved that any two irreducible representations of $K(\Hi)$ are unitarily equivalent \cite{Nai48}, and in 1951 he asked whether this property characterizes $K(\Hi)$ up to isomorphism \cite{Nai51}.  More precisely, Naimark asks the following: If $A$ is a $C^*$-algebra with only one irreducible representation up to unitary equivalence, is $A$ isomorphic to $K(\Hi)$ for some (not necessarily separable) Hilbert space $\Hi$?  This question became known as \emph{Naimark's problem}, and a hypothetical $C^*$-algebra satisfying the premise of this question but not its conclusion is called a \emph{counterexample to Naimark's problem.}

There are several perspectives with which one can view the significance of Naimark's problem.  Classically, as representations of $C^*$-algebras were studied extensively throughout the 1940s and 1950s, researchers asked to what extent the isomorphism class of a $C^*$-algebra is determined by its representation theory.  Naimark's problem may be viewed as the simplest case of this question.  Due to the GNS construction, any $C^*$-algebra has a (nonzero) irreducible representation, and hence the most basic representation theory possible for a $C^*$-algebra is when any two irreducible representations are unitarily equivalent (in other words, up to a change of Hilbert space coordinates, the $C^*$-algebra has a unique irreducible representation).  Correspondingly, the most basic $C^*$-algebra one can imagine is $K(\Hi)$ for some Hilbert space $\Hi$, and accordingly such $C^*$-algebras are called \emph{elementary $C^*$-algebras}.   Thus Naimark's problem is asking whether a $C^*$-algebra with the most basic possible representation theory must be isomorphic to the most elementary of $C^*$-algebras.


From the modern standpoint, one may also view Naimark's problem as an early inquiry into the classification of $C^*$-algebras --- one that predates the first steps of Elliott's classification program by 25 years.  Indeed, in modern language, Naimark's question is tantamount to asking whether a (not necessarily separable) $C^*$-algebra that is Morita equivalent to the compact operators on some Hilbert space must be isomorphic to $K(\Hi)$ for some (not necessarily separable) Hilbert space $\Hi$.

In the years following Naimark's proposal of the problem, various partial solutions were obtained.  In 1951, almost immediately after the problem was posed, Kaplansky showed that Naimark's question has an affirmative answer for GCR $C^*$-algebras (today more commonly known as \emph{Type I $C^*$-algebras}) \cite[Theorem~7.3]{Kap}.  Two years later, in 1953, A.~Rosenberg proved that Naimark's problem has an affirmative answer for separable $C^*$-algebras \cite[Theorem~4]{Ros}.  In 1960 Fell, building off ideas of Kaplansky, showed that any two irreducible representations of a Type I $C^*$-algebra with equal kernels must be unitarily equivalent~\cite{Fell60}. That same year, Dixmier proved a partial converse: a \emph{separable} $C^*$-algebra that is not Type I necessarily has unitarily \emph{inequivalent} representations whose kernels are equal~\cite{Dix60}. (In fact, in 1961 Glimm showed that a separable $C^*$-algebra that is not Type I has \emph{uncountably many} inequivalent irreducible representations~\cite{Glimm61}.) Dixmier's result, combined with Kaplansky's affirmative answer to Naimark's problem for Type~I $C^*$-algebras, recovered A.~Rosenberg's 1953 result. 

Despite this surge of activity in the years immediately following the question's proposal, very little progress was accomplished on Naimark's problem over the following 40 years.  The next major accomplishment came in 2004 when Akemann and Weaver used Jensen's $\diamondsuit$~axiom (pronounced ``diamond axiom"), a combinatorial principle known to be independent of ZFC, to construct a counterexample to Naimark's problem that is generated by $\aleph_1$ elements~\cite{AW04}.  This shows that, at the very least, one cannot prove an affirmative answer to Naimark's problem within ZFC alone.  In fact, Akemann and Weaver showed more: they proved that the existence of an $\aleph_1$-generated counterexample is independent of ZFC.  Whether Naimark's problem itself is independent of ZFC remains unknown.

Akemann and Weaver's result suggests that there are set-theoretic obstructions to obtaining an answer to Naimark's problem in its most general form. In light of this, it is reasonable to consider restrictions of the problem to particular types of $C^*$-algebras and to ask whether there is an affirmative answer to the problem for certain subclasses of $C^*$-algebras.  We do so in this paper for the class of graph $C^*$-algebras.  In particular, we show that Naimark's problem has an affirmative answer for (not necessarily separable) AF graph $C^*$-algebras as well as for $C^*$-algebras of graphs in which each vertex emits a countable number of edges.

It is an elementary result that a $C^*$-algebra with a unique representation up to unitary equivalence must be simple.  (See Lemma~\ref{Naimark-implies-simple-lem} of this paper for a proof.)  There is also a well-known dichotomy for simple graph $C^*$-algebras:  If the $C^*$-algebra of a (not necessarily countable) graph is simple, then the $C^*$-algebra is either AF or purely infinite.  Consequently, the results of this paper are close to establishing an answer to Naimark's problem for all graph $C^*$-algebras.  Specifically, our results show that if a graph $C^*$-algebra counterexample to Naimark's problem exists, it must be a simple purely infinite $C^*$-algebra of a graph containing a vertex that emits an uncountable numebr of edges.  Unfortunately, we are unable to determine if such a counterexample exists within the class of graph $C^*$-algebras.  Indeed, at the time of this writing it is unknown whether it is possible for any simple purely infinite $C^*$-algebra to be a counterexample to Naimark's problem.

\smallskip

\noindent \textsc{Convention:} We use the term AF-algebra to mean a $C^*$-algebra that is a direct limit of a directed system (not necessarily a directed sequence) of finite-dimensional $C^*$-algebras.  In particular, this allows for AF-algebras that are non-separable.

\smallskip

\noindent \textsc{Countable and Uncountable Graphs:}  It is fairly standard for papers on graph $C^*$-algebras to impose the standing hypothesis that all graphs are countable, despite the fact this hypothesis may not be explicitly stated in individual results.  Countability of the graph ensures that the associated $C^*$-algebra is separable, which is a common hypothesis imposed in $C^*$-algebra theory.  While separability is needed in a small number of graph $C^*$-algebra theorems (e.g., to apply $K$-theory classification), for most results it is unnecessary.  Consequently, when working with uncountable graphs, one must often go through proofs of individual results (and the results they reference) to determine whether the countability of the graph is needed, or even used.  In this paper we will need to apply four well-known results proven in papers where the graphs were assumed to be countable: (1) the simplicity of a graph $C^*$-algebra is equivalent to the graph being cofinal and satisfying Condition~(L) (which is also equivalent to the graph  having no proper nontrivial saturated hereditary subsets and satisfying Condition~(L)); proven in \cite[Theorem~4]{Pat} and \cite[Theorem~12]{Szy}), (2) a graph $C^*$-algebra is a limit of finite-dimensional $C^*$-algebras if and only if the graph has no cycles; proven in \cite[\S5.4]{RS}, (3) if $E$ is a graph and $H$ is a hereditary subset of $E$, then $C^*(E_H)$ is Morita equivalent to $I_H$; proven in \cite[Proposition~3.4]{BHRS}, and (4) the Cuntz-Krieger Uniqueness Theorem for relative graph $C^*$-algebras; proven in \cite[Theorem~3.11]{MT}.  In all four of these results the countability hypothesis is unnecessary, and the same proofs go through for uncountable graphs.

\section{Preliminaries}

\begin{definition}
A \emph{representation of $A$} is a $*$-algebra homomorphism $\pi : A \to B(\Hi)$.  A subspace $\mathcal{V} \subseteq \Hi$ is \emph{invariant (for $\pi$)} if $\pi(a)\mathcal{V} \subseteq \mathcal{V}$ for all $a \in A$.  A representation is \textit{irreducible} if its only closed invariant subspaces are $\{ 0 \}$ and $\Hi$.
\end{definition}

\begin{definition}
Two representations $\pi : A \to B(\Hi_\pi)$ and $\rho : A \to B(\Hi_\rho)$ of $A$ are \textit{unitarily equivalent}, denoted $\pi \sim_u \rho$, if there is a unitary operator $U: \Hi_\pi \rightarrow \Hi_\rho$ 
such that $\pi(a)= U^* \rho(a) U$ for all $a \in A.$ 
\end{definition}

It is straightforward to verify that unitary equivalence of representations is an equivalence relation.  In addition, it follows from the GNS construction that every $C^*$-algebra has a nonzero irreducible representation.  (See \cite[Theorem~A.14, p.210] {RW} for a statement and proof.)

\begin{definition}
We say a $C^*$-algebra \emph{has a unique irreducible representation up to unitary equivalence} if any two irreducible representations of the $C^*$-algebra are unitarily equivalent.
\end{definition}

Note that since every $C^*$-algebra has a nonzero irreducible representation, there is no need to hypothesize the existence of an irreducible representation in the above definition.

An ideal $I \triangleleft A$ is called a \emph{primitive ideal} if $I = \ker \pi$ for some irreducible representation $\pi : A \to B(\Hi)$.  One can easily see that if $\pi$ and $\rho$ are unitarily equivalent representations, then $\ker \pi = \ker \rho$.  Thus a $C^*$-algebra with a unique irreducible representation up to unitary equivalence has a unique primitive ideal.  The following straightforward lemma shows that in this case the primitive ideal is zero, and moreover any such $C^*$-algebra is simple.

\begin{lemma} \label{Naimark-implies-simple-lem}
If $A$ is a $C^*$-algebra with a unique irreducible representation up to unitary equivalence, then $A$ is simple.
\end{lemma}

\begin{proof}
It is a standard result that every closed proper ideal is equal to the intersection of all primitive ideals containing it.  (See \cite[Proposition~A.17, p.212]{RW} for a statement and self-contained proof of this result.)  Since $A$ has only one irreducible representation up to unitary equivalence, $A$ has a unique primitive ideal $I$.  Thus every closed proper ideal of $A$ must equal $I$.  Since $\{ 0 \}$ is a closed proper ideal of $A$, it follows that any closed proper ideal of $A$ is equal to $\{ 0 \}$.  In particular, $A$ is simple.
\end{proof}

A \emph{graph} $E: = (E^0, E^1, r, s)$ consists of a set of vertices $E^0$, a set of edges $E^1$, and maps $r: E^1 \to E^0$ and $s : E^1 \to E^0$ identifying the range and sources of each edge.  Throughout this paper we do not make any assumptions of the cardinality of our graphs, and in particular, we do not require the vertex or edge sets of our graphs to be finite or countable. 

A vertex $v \in E^0$ is a \emph{sink} if $s^{-1}(v) = \emptyset$.  A vertex $v \in E^0$ is an \emph{infinite emitter} if $s^{-1}(v)$ is infinite.  A \emph{singular vertex} is a vertex that is either a sink or an infinite emitter.  A \emph{regular vertex} is a vertex that is not a singular vertex; equivalently: a vertex $v$ is regular if and only if $s^{-1}(v)$ is a finite and nonempty set.  A graph is called \emph{row-finite} if for every $v \in E^0$ the set $s^{-1}(v)$ is finite (and possibly empty).  A graph is called \emph{row-countable} if for every $v \in E^0$ the set $s^{-1}(v)$ is countable (and possibly empty).  A graph $E = (E^0, E^1, r, s)$ is \emph{finite} if both $E^0$ and $E^1$ are finite sets.  A graph $E = (E^0, E^1, r, s)$ is \emph{countable} if both $E^0$ and $E^1$ are countable sets.

A \emph{path} $e_1 \ldots e_n$ in a graph $E$ consists of a finite list of edges $e_1, \ldots, e_n \in E^1$ satisfying $r(e_i) = s(e_{i+1})$ for all $1 \leq i \leq n-1$, and we say such a path has length $|\alpha| := n$.  We consider vertices to be paths of length zero, and edges to be paths of length one.  We also let $E^n$ denote the set of paths of $E$ of length $n$, and the let $E^* := \bigcup_{n=0}^\infty E^n$ denote the set of all paths of $E$.  We extend $r$ and $s$ to $E^*$ in the obvious way: If $\alpha := e_1 \ldots e_n \in E^*$, then $s(\alpha) := s(e_1)$ and $r(\alpha) := r(e_n)$.

An \emph{infinite path} $e_1 e_2 \ldots$ consists of a sequence of edges $e_1, e_2, \ldots \in E^1$ with $r(e_i) = s(e_{i+1})$ for all $i \in \N$.  We let $E^\infty$ denote the set of all infinite paths in $E$, and we extend the map $s$ to $E^\infty$ in the obvious way: If $\alpha := e_1  e_2 \ldots \in E^\infty$, then $s(\alpha) := s(e_1)$.

A \emph{cycle} is a path $\alpha \in E^*$ such that $s(\alpha) = r(\alpha)$.  If $\alpha := e_1 \ldots e_n$ is a cycle, an \emph{exit} for $\alpha$ is an edge $f \in E^1$ such that $s(f) = s(e_i)$ and $f \neq e_i$ for some $1 \leq i \leq n$.  A graph is said to satisfying Condition~(L) if every cycle in $E$ has an exit.

If $v, w \in E^0$, we say \emph{$v$ can reach $w$}, written $v \geq w$, if there exists a path $\alpha \in E^*$ with $s(\alpha) = v$ and $r(\alpha) = w$.  A graph is called \emph{cofinal} if whenever $v \in E^0$ and $\alpha := e_1 e_2 \ldots \in E^\infty$, then $v \geq s(e_i)$ for some $i \in \N$.

A subset $H \subset E^0$ is called \emph{hereditary} if whenever $e \in E^1$ and $s(e) \in H$, then $r(e) \in H$.  A hereditary subset $H$ is called \emph{saturated} if whenever $v$ is a regular vertex and $r(s^{-1}(v)) \subseteq H$, then $v \in H$.

\begin{definition}
If $E := (E^0, E^1, r, s)$ is a graph, a \emph{Cuntz-Krieger $E$-family} is a collection of elements $\{ s_e, p_v : e \in E^1, v \in E^0 \}$ in a $C^*$-algebra such that $\{ p_v : v \in E^0 \}$ is a collection of mutually orthogonal projections and $\{ s_e : e \in E^1 \}$ is a collection of partial isometries with mutually orthogonal ranges satisfying  the \emph{Cuntz-Krieger relations}:
\begin{itemize}
\item[(CK1)] $s_e^* s_e = p_{r(e)}$ for all $e \in E^1$,
\item[(CK2)] $s_e s_e^* \leq p_{s(e)} $ for all $e \in E^1$, and 
\item[(CK3)] $p_v = \sum_{s(e)=v} s_es_e^*$ whenever $v \in E^0$ is a regular vertex.
\end{itemize}
The \emph{graph $C^*$-algebra} $C^*(E)$ is the universal $C^*$-algebra generated by a Cuntz-Krieger $E$-family.
\end{definition}

If $\alpha := e_1 \ldots e_n \in E^*$, we define $s_\alpha := s_{e_1} \ldots s_{e_n}$, and when $\alpha = v \in E^0$ we interpret this as $s_\alpha := p_v$.  One can use the Cuntz-Krieger relations to show that $C^*(E) = \clspan \{ s_\alpha s_\beta^* : \alpha, \beta \in E^* \}$.  Moreover, $C^*(E)$ is separable if and only if $E$ is a countable graph.  Indeed, when $E$ is countable 
$$\left\{ \sum_{k=1}^n (a_k + i b_k) s_{\alpha_k} s_{\beta_k}^* : n \in \N \text{ and } \alpha_k, \beta_k \in E^*,  a_k, b_k \in \mathbb{Q} \text{ for } 1 \leq k \leq n \right\}$$
is a countable dense subset of $C^*(E)$.

There are numerous results relating the structure of $E$ to the structure of $C^*(E)$, and we state (with reference) a few that we shall use throughout this paper.

\begin{theorem}  \label{simple-theorem}
(\cite[Theorem~4]{Pat}, \cite[Theorem~12]{Szy})
Let $E$ be a directed graph.  Then the following three conditions are equivalent:
\begin{itemize}
\item[(1)] $C^*(E)$ is simple.
\item[(2)] $E$ satisfies Condition~(L), $E$ is cofinal, and every vertex of $E$ can reach every singular vertex of $E$.
\item[(3)] $E$ satisfies Condition~(L) and the only saturated hereditary subsets of $E$ are $E^0$ and 
$\emptyset$.
\end{itemize}
\end{theorem}

We mention that while the papers \cite{Pat, Szy} impose a standing hypothesis that the graphs they consider are countable, this hypothesis in unnecessary for their proofs of the results stated in Theorem~\ref{simple-theorem}.

In this paper, we call a $C^*$-algebra an \emph{AF-algebra}, or say the $C^*$-algebra is \emph{AF} (short for approximately finite-dimensional), if it is the direct limit of finite-dimensional $C^*$-algebras.  Note that this differs from the standard usage of the term, which typically requires a sequential direct limit and therefore implies the limit is separable.  Our notion of AF coincides with the usual definition for separable $C^*$-algebras, but also allows for nonseparable AF-algebras, which are direct limits of general directed families of finite-dimensional $C^*$-algebras.  The following result gives a very nice characterization of AF graph $C^*$-algebras.

\begin{theorem}  \label{AF-iff-no-cycles-thm}
(\cite[\S5.4]{RS})
If $E$ is a graph, then $C^*(E)$ is AF if and only if $E$ has no cycles.
\end{theorem}

Although this result in \cite[\S5.4]{RS} is stated and proven for countable graphs, the same proof works without the countability hypothesis, showing that $C^*(E)$ is a (not necessarily sequential) direct limit of finite-dimensional $C^*$-algebras if and only if $E$ has no cycles.

If $E = (E^0, E^1, r_E, s_E)$ is a graph, a \emph{subgraph} of $E$ is a graph $F = (F^0, F^1, r_F, s_F)$ such that $F^0 \subseteq E^0$, $F^1 \subseteq E^*$, and $r_F$ and $s_F$ are restrictions of the range and source maps that $r_E$ and $s_E$ induce on $E^*$.

If $H$ is a hereditary subset (that is not necessarily saturated) of the graph $E = (E^0, E^1, r, s)$, the \emph{restriction of $E$ to $H$} is the graph $E_H := ( E_H^0, E_H^1, r_H, s_H)$ with vertex set $E_H^0 := H$, edge set $E_H^1 := s^{-1}(H)$, and range and source maps $r_H := r|_{E_H^1}$ and $s_H := s|_{E_H^1}$.  In addition, we let $I_H$ denote the closed two-sided ideal in $C^*(E)$ generated by $\{ p_v : v \in H \}$.

\begin{theorem} \label{hereditary-restriction-thm} 
(\cite[Proposition~3.4]{BHRS})
If $E$ is a graph and $H$ is a hereditary subset of $E$, then $C^*(E_H)$ is Morita equivalent to $I_H$.
\end{theorem}

Although the above result is stated for countable graphs in \cite[Proposition~3.4]{BHRS}, the countability hypothesis in unnecessary, and the same proof works for general (not necessarily countable) graphs.

\begin{definition}
(\cite[\S3]{MT})
If $E := (E^0, E^1, r, s)$ is a graph, let $E^0_\textnormal{reg}$ denote the regular vertices of $E$, and let $S \subseteq E^0_\textnormal{reg}$.  A \emph{Cuntz-Krieger $(E,S)$-family} is a collection of elements $\{ s_e, p_v : e \in E^1, v \in E^0 \}$ in a $C^*$-algebra such that $\{ p_v : v \in E^0 \}$ is a collection of mutually orthogonal projections and $\{ s_e : e \in E^1 \}$ is a collection of partial isometries with mutually orthogonal ranges satisfying  the \emph{Cuntz-Krieger relations}:
\begin{itemize}
\item[(CK1)] $s_e^* s_e = p_{r(e)}$ for all $e \in E^1$,
\item[(CK2)] $s_e s_e^* \leq p_{s(e)} $ for all $e \in E^1$, and 
\item[(CK3)] $p_v = \sum_{s(e)=v} s_es_e^*$ whenever $v \in S$.
\end{itemize}
The \emph{relative graph $C^*$-algebra} $C^*(E,S)$ is the $C^*$-algebra generated by a universal Cuntz-Krieger $(E,S)$-family.
\end{definition}

Observe that if $S = E^0_\textnormal{reg}$, then $C^*(E, E^0_\textnormal{reg})$ is exactly the graph $C^*$-algebra $C^*(E)$.  If $S = \emptyset$, then $C^*(E,\emptyset)$ is called the Toeplitz algebra of $E$, and often denoted $TC^*(E)$.

If $C^*(E,S)$ is a relative graph $C^*$-algebra and  $\{ s_e, p_v : e \in E^1, v \in E^0 \}$ is a generating Cuntz-Krieger $(E,S)$-family in $C^*(E,S)$, then for any $v \in E^0_\textnormal{reg} \setminus S$, we call $q_v := p_v - \sum_{s(e) = v} s_es_e^*$ the \emph{gap projection} at $v$.

If $I$ is the ideal generated by $\{ q_v : v \in E^0_\textnormal{reg} \setminus S \}$, then $C^*(E, S) / I \cong C^*(E)$, and hence the graph $C^*$-algebra is a quotient of the relative graph $C^*$-algebra $C^*(E,S)$.

In addition, whenever $E$ is a graph and $S \subseteq E^0_\textnormal{reg}$ there exists a graph $E_S$ such that $C^*(E_S)$ is isomorphic to $C^*(E,S)$.  Thus every relative graph $C^*$-algebra is isomorphic to a graph $C^*$-algebra (of a possibly different graph).  

Furthermore, we have the following Cuntz-Krieger Uniqueness Theorem for relative graph $C^*$-algebras.

\begin{theorem} 
(\cite[Theorem~3.11]{MT})
Let $E$ be a graph, let $S \subseteq E^0_\textnormal{reg}$, and let $\phi : C^*(E,S) \to A$ be a homomorphism from $C^*(E,S)$ into a $C^*$-algebra $A$.  If $\{ s_e, p_v : e \in E^1, v \in E^0 \}$ is a generating Cuntz-Krieger $(E,S)$-family in $C^*(E,S)$ and the following three conditions hold
\begin{itemize}
\item[(1)] $E$ satisfies Condition~(L),
\item[(2)] $\phi (p_v) \neq 0$ for all $v \in E^0$, and
\item[(3)] $\phi (q_v) \neq 0$ for all $v \in E^0_\textnormal{reg} \setminus S$,
\end{itemize}
then $\phi$ is injective.
\end{theorem}

\section{Structure results for graph $C^*$-algebra with a unique irreducible representation}

\begin{lemma} \label{M-2-infinity-graph-lem}

If $F$ is the graph
\begin{equation} \label{M2-graph-eq}
\xymatrix{
w_1 \ar@/^/[r]^{e_1} \ar@/_/[r]_{f_1} & w_2 \ar@/^/[r]^{e_2} \ar@/_/[r]_{f_2} & w_3 \ar@/^/[r]^{e_3} \ar@/_/[r]_{f_3} &  \cdots
}
\end{equation}
then $C^*(F)$ contains a full corner isomorphic to the UHF-algebra $M_{2^\infty}$, and $C^*(F)$ is not a Type~I $C^*$-algebra.
\end{lemma}

\begin{proof}
Consider the corner $p_{w_1} C^*(F) p_{w_1}$.  Since $F$ has no cycles, the ideals of $C^*(F)$ are in bijective correspondence with the saturated hereditary subsets of $F$.  Since $p_{w_1} \in p_{w_1} C^*(F) p_{w_1}$, and any hereditary subset containing $w_1$ must equal $F^0$, we may conclude that any ideal containing $p_{w_1} C^*(F) p_{w_1}$ is equal to $C^*(F)$.  Thus $p_{w_1} C^*(F) p_{w_1}$ is a full corner of $C^*(F)$.

If $\{ s_e, p_v : e \in F^1, v \in F^0 \}$ is a generating Cuntz-Krieger $F$-family, then we see that 
$$p_{w_1} C^*(F) p_{w_1} = \clspan \{ s_\alpha s_\beta^* : s(\alpha) = s(\beta) = w_1 \text{ and } r(\alpha) = r(\beta) \}.$$
For each $n \in \N \cup \{ 0 \}$, define 
$$E_n := \{ \alpha \in F^* : s(\alpha) = w_1 \text{ and } r(\alpha) = w_n \}$$
and let 
$$A_n := \clspan \{ s_\alpha s_\beta^* : \alpha, \beta \in E_n \}.$$  Then we see that each $A_n$ is a $C^*$-subalgebra of $C^*(F)$, $A_0 \subseteq A_1 \subseteq A_2 \subseteq \ldots$, and 
$$C^*(F) = \overline{\bigcup_{n=0}^\infty A_n}.$$
For each $n \in \N \cup \{ 0 \}$, we see that $\{ s_\alpha s_\beta^* : \alpha,\beta \in E_n\}$ is a set of matrix units, and since $|E_n| = 2^n$ it follows that $A_n \cong M_{2^n} (\mathbb{C})$.  Furthermore, for each $s_\alpha s_\beta^*$ with $\alpha, \beta \in E_n$, we see that 
$$s_\alpha s_\beta^* = s_\alpha p_{w_n} s_\beta^* =  s_\alpha (s_{e_n} s_{e_n}^* + s_{f_n} s_{f_n}^* ) s_\beta^* = s_{\alpha e_n} s_{\beta e_n}^* + s_{\alpha f_n} s_{\beta f_n}^*.$$
Hence if for all $n \in \N \cup \{ 0 \}$ we identify $A_n$ with $M_{2^n} (\mathbb{\C})$ via an isomorphism, then for each $n$ the inclusion map $A_n \hookrightarrow A_{n+1}$ may be identified with the map $x \mapsto \left( \begin{smallmatrix} x & 0 \\ 0 & x \end{smallmatrix} \right)$.  Thus $C^*(F) = \overline{\bigcup_{n=0}^\infty A_n}$ is isomorphic to the UHF-algebra $M_{2^\infty}$.

Finally, it follows from \cite[Theorem~6.5.7, p.211]{Ped} that $M_{2^\infty}$ is not Type~I.  Since any $C^*$-subalgebra of a Type~I $C^*$-algebra is Type~I \cite[Theorem~6.2.9, p.199]{Ped}, we conclude that $C^*(F)$ is not Type~I. 
\end{proof}

\begin{proposition} \label{Naimark-hereditary-prop}
Let $E$ be a row-countable directed graph such that $C^*(E)$ has a unique irreducible representation up to unitary equivalence.  If $v \in E^0$ and we define $H(v) := \{ w \in E^0 : v \geq w \}$, then $E_{H(v)}$ is a countable graph, $C^*(E_{H(v)})$ is Morita equivalent to $C^*(E)$, and $C^*(E_{H(v)}) \cong K(\Hi)$ for some separable Hilbert space $\Hi$.
\end{proposition}

\begin{proof}
We see that $H(v) := \{ w \in E^0 : v \geq w \}$ is a hereditary subset of $E$.   In addition, if we let $H_0 = \{ v \}$ and inductively define $H_{n+1} := r(s^{-1}(H_n))$, then one can easily verify that $H(v) = \bigcup_{n=0}^\infty H_n$.  Since $H_0$ is finite and $E$ is row-countable, an inductive argument shows that $H_n$ is countable for all $n \in \N$. Hence $E_{H(v)}^0 := H(v) = \bigcup_{n=0}^\infty H_n$ is countable, and $E_{H(v)}^1 := s^{-1}(H_v)$ is countable.  Thus $E_{H(v)} := (E_{H(v)}^0, E_{H(v)}^1, r|_{H(v)}, s|_{H(v)})$ is a countable graph.

It follows from Theorem~\ref{hereditary-restriction-thm} that $C^*(E_{H(v)})$ is Morita equivalent to $I_{H(v)}$.  Since $C^*(E)$ has a unique irreducible representation up to unitary equivalence, it follows from Lemma~\ref{Naimark-implies-simple-lem} that $C^*(E)$ is simple.  Because $I_{H(v)}$ is a nonzero ideal of $C^*(E)$, it follows that $I_{H(v)} = C^*(E)$.  Thus $C^*(E_{H(v)})$ is Morita equivalent to $C^*(E)$

Finally, since $C^*(E)$ has a unique irreducible representation up to unitary equivalence and $C^*(E_{H(v)})$ is Morita equivalent to $C^*(E)$, we may conclude that $C^*(E_{H(v)})$ has a unique irreducible representation up to unitary equivalence.  (This is because the Rieffel correspondence provides a bijective correspondence between representations of Morita equivalent $C^*$-algebras that preserves irreducibility and unitary equivalence.)  Moreover, the fact that $E_{H(v)} := (E_{H(v)}^0, E_{H(v)}^1, r|_{H(v)}, s|_{H(v)})$ is a countable graph implies $C^*(E_{H(v)})$ is a separable $C^*$-algebra.  Hence by  Rosenberg's Theorem  \cite[Theorem~4]{Ros}, $C^*(E_{H(v)}) \cong K(\Hi)$ for a separable Hilbert space $\Hi$.  
\end{proof}

\begin{proposition} \label{three-equivalent-Naimark-prop}
Let $E$ be a directed graph such that $C^*(E)$ has a unique irreducible representation up to unitary equivalence.  Then the following are equivalent:
\begin{itemize}
\item[(1)] $E$ is row-countable.
\item[(2)] $E$ is row-finite.
\item[(3)] $C^*(E)$ is AF.
\end{itemize}
\end{proposition}

\begin{proof}
\noindent $(2) \implies (1)$:  This is immediate from the definitions.

\noindent $(1) \implies (3)$:  Let $v \in E^0$, and set $H(v):= \{ w \in E^0 : v \geq w \}$.  Then $H(v)$ is a hereditary subset of $E$, and by Proposition~\ref{Naimark-hereditary-prop} $C^*(E_{H(v)}) \cong K(\Hi)$ for some separable Hilbert space $\Hi$.  Consequently $C^*(E_{H(v)})$ is AF, and Theorem~\ref{AF-iff-no-cycles-thm} implies the graph $E_{H(v)}$ has no cycles.  Hence $E$ has no cycles with vertices in $H(v)$.  Furthermore, $C^*(E)$ is simple by Lemma~\ref{Naimark-implies-simple-lem}, and thus Theorem~\ref{simple-theorem} implies that $E$ is cofinal.  Since vertices in the hereditary set $H(v)$ cannot reach cycles containing vertices in $E^0 \setminus H(v)$, we may conclude that $E$ has no cycles.  Thus Theorem~\ref{AF-iff-no-cycles-thm} implies $C^*(E)$ is AF.

\noindent $(3) \implies (2)$:  Since $C^*(E)$ is AF,  Theorem~\ref{AF-iff-no-cycles-thm} implies the graph $E$ has no cycles.  In addition, since $C^*(E)$ has a unique irreducible representation up to unitary equivalence, Lemma~\ref{Naimark-implies-simple-lem} implies that $C^*(E)$ is simple.  It then follows from Theorem~\ref{simple-theorem} that every vertex of $E$ can reach every singular vertex of $E$.  Let $v \in E^0$, and suppose $v$ is not a sink.  Then there exists $e \in s^{-1}(v)$.  Since $E$ has no cycles, it follows that $r(e)$ cannot reach $v$.  But this implies that $v$ is not a singular vertex, and hence $v$ emits a finite number of edges.  Since every vertex of $E$ that is not a sink emits a finite number of edges, $E$ is row-finite.
\end{proof}

\begin{proposition} \label{forbidden-subgraph-prop}
Let $E$ be a directed graph such that $C^*(E)$ has a unique irreducible representation up to unitary equivalence.  If $C^*(E)$ is AF, then $E$ does not contain a subgraph of the form 
\begin{equation} \label{subgraph-eq}
\xymatrix{
v_1 \ar@{-->}@/^/[rr]^{\beta_1} \ar@{-->}@/_/[rr]_{\alpha_1} & & v_2  \ar@{-->}@/^/[rr]^{\beta_2} \ar@{-->}@/_/[rr]_{\alpha_2} & & v_3  \ar@{-->}@/^/[rr]^{\beta_3} \ar@{-->}@/_/[rr]_{\alpha_3} & & \cdots
}
\end{equation}
where $v_1, v_2, \ldots$ are distinct vertices and $\alpha_1, \beta_1, \alpha_2, \beta_2, \ldots$ are distinct paths.
\end{proposition}

\begin{proof}
For the sake of contradiction, suppose that $E$ has a subgraph of the form in \eqref{subgraph-eq}, and use the labeling of vertices and paths listed in \eqref{subgraph-eq}.  Since $C^*(E)$ is AF, Proposition~\ref{three-equivalent-Naimark-prop} implies $E$ is row-finite.  If $H(v_1):= \{ w \in E^0 : v_1 \geq w \}$, then $H(v_1)$ is a hereditary subset of $E$ and Proposition~\ref{Naimark-hereditary-prop} implies $C^*(E_{H(v)}) \cong K(\Hi)$ for a separable Hilbert space $\Hi$.  Thus $C^*(E_{H(v)})$ is a Type~I $C^*$-algebra.  Furthermore, since $v_1$ can reach every vertex on each path $\alpha_i$ and each path $\beta_i$ for all $i \in \N$, we conclude that the graph $E_{H(v)}$ has a subgraph of the form in  \eqref{subgraph-eq}.

Let $\{ s_e, p_v : e \in E_{H(v)}^1, v \in E_{H(v)}^0 \}$ be a generating Cuntz-Krieger $E_{H(v)}$-family in $C^*(E_{H(v)})$, and consider the set $\{ p_{v_i}, s_{\alpha_i}, s_{\beta_i} \}_{i=1}^\infty$.   Also let $F$ be the graph
\begin{equation} \label{F-graph-eq}
\xymatrix{
w_1 \ar@/^/[r]^{e_1} \ar@/_/[r]_{f_1} & w_2 \ar@/^/[r]^{e_2} \ar@/_/[r]_{f_2} & w_3 \ar@/^/[r]^{e_3} \ar@/_/[r]_{f_3} &  \cdots
}
\end{equation}
and let 
$$S := \{ w_i : i \in \N \text{ and } p_{v_i} = s_{\alpha_i} s_{\alpha_i}^* + s_{\beta_i} s_{\beta_i}^* \}.$$
Then $\{ p_{v_i}, s_{\alpha_i}, s_{\beta_i} \}_{i=1}^\infty$ is an $(F,S)$-family in $C^*(E_{H(v)})$, and there exists a homomorphism $\phi : C^*(F,S) \to C^*(E_{H(v)})$ (where $C^*(F,S)$ denotes the relative graph $C^*$-algebra of $F$ with the (CK3) relation imposed at the vertices in $S$).  We observe that $F$ has no cycles, and whenever $i \in \N$ with $w_i \notin S$, then the gap projection $p_{v_i} - s_{\alpha_i} s_{\alpha_i}^* - s_{\beta_i} s_{\beta_i}^* \neq 0$ whenever $v_i \notin S$, it follows from the Cuntz-Krieger Uniqueness Theorem for relative graph $C^*$-algebras that $\phi$ is injective.  Hence $\im \phi$ is a $C^*$-subalgebra of $C^*(E_{H(v)})$ isomorphic to $C^*(F,S)$.

It follows from Proposition~\ref{M-2-infinity-graph-lem} that $C^*(F)$ is not a Type~I $C^*$-algebra.  Since $C^*(F)$ is a quotient of $C^*(F,S)$, and all quotients of Type~I $C^*$-algebras  are Type~I \cite[Theorem~6.2.9, p.199]{Ped}, it follows that $C^*(F,S)$ is not a Type~I $C^*$-algebra.  Thus $\im \phi \cong C^*(F,S)$ is a $C^*$-subalgebra of $C^*(E_{H(v)})$ that is not Type~I, and since all $C^*$-subalgebras of Type~I $C^*$-algebras are Type~I \cite[Theorem~6.2.9, p.199]{Ped}, it follows that $C^*(E_{H(v)})$  is not a Type~I $C^*$-algebra.  But this contradicts the fact that $C^*(E_{H(v)}) \cong K(\Hi)$.
\end{proof}

\begin{proposition} \label{Naimark-AF-dichotomy-prop}
Let $E$ be a directed graph such that $C^*(E)$ is AF and has a unique irreducible representation up to unitary equivalence.  Then one of two distinct possibilities must occur: Either
\begin{itemize}
\item[(1)] $E$ has exactly one sink and no infinite paths; or
\item[(2)] $E$ has no sinks and $E$ contains an infinite path $\alpha := e_1 e_2 \ldots$ with $s^{-1}( s(e_i)) = \{ e_i \}$ for all $i \in \N$. 
\end{itemize}
\end{proposition}

\begin{proof}
Since $C^*(E)$ has a unique irreducible representation up to unitary equivalence, it follows from Lemma~\ref{Naimark-implies-simple-lem} that $C^*(E)$ is simple, and it follows from Theorem~\ref{simple-theorem} that $E$ is cofinal, satisfies Condition~(L), and every vertex of $E$ can reach every singular vertex of $E$.

The fact that every vertex of $E$ can reach every singular vertex of $E$ implies that $E$ has at most one sink.  If $E$ has one sink, then the cofinality of $E$ implies that $E$ has no infinite paths (since a sink cannot reach a vertex on the infinite path), and hence we are in situation (1) of the proposition.  

If $E$ has no sinks, then $E$ must contain an infinite path $f_1 f_2 \ldots$.  To show that we are in situation (2) it suffices to show that there exists $N \in \N$ such that $s^{-1}(s(f_i)) = \{ f_i \}$ for all $i \geq N$.  (For then we can take $e_i := f_{N+i}$, and $e_1 e_2 \ldots$ is the desired path.)

Suppose to the contrary that the infinite path $f_1 f_2 \dots$ does not have our desired property.  This means that for each $k \in \N$ there exists $n \geq k$ such that $s^{-1}(s(f_n))$ contains an element different from $f_n$.  For convenience of notation, we shall set $v_i := s(f_i)$ for all $i \in \N$.

We shall describe an inductive construction to produce a subgraph: To begin, choose a natural number $n_1$ such that $s^{-1}(f_{n_1})$ contains an element $g$ different from $f_{n_1}$.  By cofinality there exists a path $\mu$ with $s(\mu) = r(g)$ and $r(\mu) = s(f_{n_2})$ for some $n_2 \in \N$.  Since $C^*(E)$ is AF, it follows from Theorem~\ref{AF-iff-no-cycles-thm} that $E$ has no cycles, and hence $n_2 > n_1$.  Moreover, by the defining property of the path $f_1f_2 \ldots$ we may, after possibly extending $\mu$ along this path, assume that $s^{-1}(f_{n_2})$ contains an element $g'$ different from $f_{n_2}$.   If we let $\alpha_1 := f_{n_1} f_{n_1+1} \ldots f_{n_2-1}$ and $\beta_1 := g \mu$, we have produced a subpath
$$ \xymatrix{
v_{n_1} \ar@{-->}@/^/[rr]^{\beta_1} \ar@{-->}@/_/[rr]_{\alpha_1} & &v_{n_2}
}
$$
with $n_1 < n_2$ and the property that $v_{n_2} = s^{-1}(f_{n_2})$ contains an element $g'$ different from $f_{n_2}$. 

Repeating this process, we inductively construct a subgraph
$$
\xymatrix{
v_{n_1} \ar@{-->}@/^/[rr]^{\beta_1} \ar@{-->}@/_/[rr]_{\alpha_1} & & v_{n_2}  \ar@{-->}@/^/[rr]^{\beta_2} \ar@{-->}@/_/[rr]_{\alpha_2} & & v_{n_3}  \ar@{-->}@/^/[rr]^{\beta_3} \ar@{-->}@/_/[rr]_{\alpha_3} & & \cdots
}
$$
with $n_1 < n_2 < n_3 < \ldots$ in the vertex subscripts.

Since $C^*(E)$ has a unique irreducible representation up to unitary equivalence, Proposition~\ref{forbidden-subgraph-prop} implies that $E$ does not have such a subgraph.  Hence we have a contradiction.
\end{proof}

\section{Naimark's Problem for certain graph $C^*$-algebras}

In this section we prove our two main results: (1) Naimark's Problem has an affirmative answer for the class of AF graph $C^*$-algebras, and (2) Naimark's Problem has an affirmative answer for the class of $C^*$-algebras of row-countable graphs.

If $\Hi$ is a Hilbert space, then for any $x,y \in \Hi$, we let $\Theta_{x,y} : \Hi \to \Hi$ denote the rank-one operator given by $$\Theta_{x,y}(z) := \langle y,z \rangle x.$$  Since $K(\Hi)$ is the closure of the finite-rank operators, we see that if $\beta$ is a basis for $\Hi$, then $K(\Hi) = \clspan \{ \Theta_{x,y} : x,y \in \beta \}$.

If $V : \Hi_1 \to \Hi_2$ is an isometry between Hilbert spaces, we let $\Ad_V : K(\Hi_1) \to K(\Hi_2)$ denote the homomorphism given by $\Ad_V (T) := V T V^*$.  It is straightforward to verify that $\Ad_V$ is injective and for any $x,y \in \Hi_1$ we have $\Ad_V ( \Theta_{x,y} ) = \Theta_{Vx, Vy}$.

\begin{theorem} \label{Naimark-AF-thm}
Let $E$ be a directed graph such that $C^*(E)$ is AF.  If $C^*(E)$ has a unique irreducible representation up to unitary equivalence, then $C^*(E) \cong K(\Hi)$ for some Hilbert space $\Hi$.
\end{theorem}

\begin{proof}
Throughout, let $\{s_e, p_v : e \in E^1, v \in E^0 \}$ be a generating Cuntz-Krieger $E$-family.  By Proposition~\ref{Naimark-AF-dichotomy-prop} there are two cases to consider.

\smallskip

\noindent \textsc{Case I:} $E$ has exactly one sink and no infinite paths.

Let $v_0$ denote the sink of $E$, and let $E^*(v_0) := \{ \alpha \in E^* : r(\alpha) = v_0 \}$.  Define $I_{v_0} := \clspan \{ s_\alpha s_\beta^* : \alpha, \beta \in E^*(v_0) \}$.  Since no path ending at the sink $v_0$ can be extended, for any finite paths $\alpha, \beta, \gamma, \delta$ with either $r(\alpha)=r(\beta) = v_0$ or $r(\gamma) = r(\delta)$ we have 
$$ (s_\alpha s_\beta^*) (s_\gamma s_\delta^*) := \begin{cases} s_\alpha s_\delta^* & \text{ if $\beta = \gamma$} \\ 0 & \text{ if $\beta \neq \gamma$,} \end{cases} $$
which implies that $I_{v_0}$ is an ideal, and that $\{ s_\alpha s_\beta^* : \alpha, \beta \in E^*(v_0) \}$ is a set of matrix units indexed by $E^*(v_0)$.  Hence $I_{v_0} \cong K( \Hi)$, where $\Hi := \ell^2 ( E^*(v_0) )$.  Furthermore, since $p_{v_0} \in I_{v_0}$, the ideal $I_{v_0}$ is nonzero.  By Lemma~\ref{Naimark-implies-simple-lem} $C^*(E)$ is simple, and hence $I_{v_0} = C^*(E)$.  Thus the result holds in this case.

\smallskip

\noindent \textsc{Case II:} $E$ contains an infinite path $\alpha := e_1 e_2 \ldots$ with $s^{-1}( s(e_i)) = \{ e_i \}$ for all $i \in \N$.

$$
\xymatrix{ v_1 \ar[r]^{e_1} & v_2 \ar[r]^{e_2} & v_3 \ar[r]^{e_3} & \cdots
}
$$
For convenience of notation, let $v_i := s(e_i)$, and for each $n \in \N$ define $E^*(v_n) := \{ \alpha \in E^* : r(\alpha) = v_n$.  Let $\Hi_n := \ell^2 ( E^*(v_n))$, and for each $\alpha \in E^*(v_n)$ let $\delta_\alpha \in \Hi_n$ denote the point mass function at $\alpha$, so that $\{ \delta_\alpha : \alpha \in E^*(v_n) \}$ forms an orthonormal basis for $\Hi_n$.

For each $n \in \N$ define $V_n : \Hi_n \to \Hi_{n+1}$ to be the isometry with $$V_n ( \delta_\alpha) := \delta_{\alpha e_n}$$ for each $\alpha \in \ell^2 ( E^*(v_n))$.  Also define $\Ad_{V_n} : K(\Hi_n) \to K(\Hi_{n+1})$ by $\Ad_{V_n} (T) := V_n T V_n^*$.

For each $n \in \N$ define $A_n := \{ s_\alpha s_\beta^* : \alpha, \beta \in E^*(v_n) \}$.  If we consider the generating set $\{ s_\alpha s_\beta^* : \alpha, \beta \in E^*(v_n) \}$, then for any $\beta, \gamma \in E^*(v_n)$, we have $r(\beta) = r(\gamma) = v_n$, and since $E$ has no cycles the only way for one of $\beta$ and $\gamma$ to extend the other is if $\beta = \gamma$.  Hence for any $\alpha, \beta, \gamma, \delta \in E^*(v_n)$, we have
$$ (s_\alpha s_\beta^*) (s_\gamma s_\delta^*) := \begin{cases} s_\alpha s_\delta^* & \text{ if $\beta = \gamma$} \\ 0 & \text{ if $\beta \neq \gamma$} \end{cases} $$
and $\{ s_\alpha s_\beta^* : \alpha, \beta \in E^*(v_n)$ is a set of matrix units indexed by $E^*(v_n)$.  Hence there exists an isomorphism $\phi_n : A_n \to K(\Hi_n)$ satisfying $\phi_n (\Theta_{\alpha, \beta}) = s_\alpha s_\beta^*$.  

Let $\iota_n : A_n \hookrightarrow A_{n+1}$ denote the inclusion map.  For each $n \in \N$ and for all $\alpha, \beta \in E^*(v_n)$ we have 
\begin{align*}
\phi_{n+1} &\circ \Ad_{V_n} (\Theta_{\delta_\alpha, \delta_\beta}) = \phi_{n+1} (V_n \Theta_{\delta_\alpha, \delta_\beta} V_n^*) = \phi_{n+1} ( \Theta_{V_n \delta_\alpha, V_n \delta_\beta} ) \\
&= \phi_n ( \Theta_{\delta_{\alpha e_n}, \delta_{\beta e_n}} ) = s_{\alpha e_n} s_{\beta e_n}^* =  s_\alpha s_{e_n} s_{e_n}^* s_\beta^* = s_\alpha p_{s(e_n)} s_\beta^* =   s_\alpha p_{s(e_n)} s_\beta^* \\
&= s_\alpha s_\beta^* = \iota_n \circ \phi_n ( \Theta_{\delta_\alpha, \delta_\beta}).
\end{align*}
Thus for each $n \in \N$ we have $\phi_{n+1} \circ \Ad_{V_n} =  \iota_n \circ \phi_n$ and the diagram
\begin{equation} \label{square-commute-eq}
\xymatrix{ 
K(\Hi_n) \ar[r]^<>(.4){\Ad_{V_n}} \ar[d]_{\phi_n} & K (\Hi_{n+1}) \ar[d]_{\phi_{n+1}} \\
A_n \ar[r]^{\iota_n} & A_{n+1}
}
\end{equation}
commutes.  Since the direct limit of the sequence 
$$
\xymatrix{ A_1 \ar[r]^{\iota_1} & A_2 \ar[r]^{\iota_2} & A_3 \ar[r]^{\iota_3} & \ldots}
$$
is equal to $\overline{\bigcup_{n=1}^\infty A_n}$, and since for all $n \in \N$ the map $\phi_n : A_n \to A_{n+1}$ is an isomorphism and the diagram in \eqref{square-commute-eq} commutes, we may conclude that 
\begin{equation} \label{compact-union-eq}
\varinjlim K(\Hi_n) \cong \overline{\bigcup_{n=1}^\infty A_n},
\end{equation}
where $\varinjlim K(\Hi_n)$ is the direct limit of the sequence
$$
\xymatrix{ K(\Hi_1) \ar[r]^<>(.4){\Ad_{V_1}} & K(\Hi_2) \ar[r]^<>(.4){\Ad_{V_2}} & K(\Hi_3) \ar[r]^<>(.4){\Ad_{V_3}} & \ldots}.
$$

Next we consider the set of infinite paths $E^\infty$.  For any infinite path $\mu \in E^\infty$, we must have $\mu = \alpha e_i e_{i+1} e_{i+1} \ldots$ for some $\alpha \in E^*$ and some $i \in \N$, for otherwise the vertex $v_1$ could not reach a vertex on $\mu$, contradicting the cofinality of $E$.

Define $\Hi_\infty := \ell^2 (E^\infty)$ and for $\mu \in E^\infty$ let $\delta_\mu$ denote the point mass function at $\mu$.  Then $\{ \delta_\mu : \mu \in E^\infty \}$ is an orthonormal basis for $\Hi_\infty$.  For each $n \in \N$ define an isometry $W_n : \Hi_n \to \Hi_\infty$ by $$W_n (\delta_\alpha) = \delta_{\alpha e_n e_{n+1} \ldots}.$$  For each $n \in \N$ and for any $\alpha \in E^*(v_n)$ we have
$$W_{n+1} (V_n (\delta_\alpha)) = W_{n+1} (\delta_{\alpha e_n}) = \delta_{\alpha e_n e_{n+1} e_{n+2} \ldots } = W_n (\delta_{\alpha})$$
and hence $W_{n+1} \circ V_n = W_n$ for all $n \in \N$.  

In addition, for any $n \in \N$ we define $\Ad_{W_n} : K(\Hi_n) \to K(\Hi_\infty)$ by $\Ad_{W_n} (T) := W_n T W_n^*$.  For any $T \in K(\Hi_n)$ we have
\begin{align*}
\Ad_{W_{n+1}} \circ \Ad_{V_n} (T) &= \Ad_{W_{n+1}} (V_n T V_n^*) = W_{n+1} V_n T V_n^* W_{n+1}^* \\
&= (W_{n+1} V_n) T (W_{n+1} V_n)^* = W_n T W_n^* = \Ad_{W_n} (T)
\end{align*}
so that $$\Ad_{W_{n+1}} \circ \Ad_{V_n} = \Ad_{W_n}$$ for all $n \in \N$.  By the universal property of the direct limit there exists a homomorphism $\psi : \varinjlim K(\Hi_n) \to K (\Hi_\infty )$ with $\im \Ad_{W_n} \subseteq \im \psi$ for all $n \in \N$, and furthermore, since each $\Ad_{W_n}$ is injective for all $n \in \N$, we may conclude that $\psi :  \varinjlim K(\Hi_n) \to K (\Hi_\infty )$ is injective.

Moreover, for any $\mu, \nu \in E^\infty$, we may write $\mu = \alpha e_j e_{j+1} \ldots$ and $\nu = \beta e_j e_{j+1} \ldots$ for some $j \in \N$ and some $\alpha, \beta \in E^*(v_j)$, from which it follows that
\begin{align*}
\Theta_{\delta_\mu, \delta_\nu} &= \Theta_{\delta_{\alpha e_j e_{j+1} \ldots}, \delta_{\beta e_j e_{j+1} \ldots}} =
\Theta_{W_j(\delta_{\alpha}), W_j(\delta_{\beta})} 
= W_j \Theta_{\alpha, \beta} W_j^* \\
&= \Ad_{W_j} (\Theta_{\delta_\alpha, \delta_\beta}) \in \im \Ad_{W_j} \subseteq \im \psi.
\end{align*}
Hence $\{ \Theta_{\delta_\mu, \delta_\nu} : \mu, \nu \in E^\infty \} \subseteq \im \psi$, so that $\im \psi = K(\Hi_\infty)$, and $\psi$ is surjective.  Therefore $\psi : \varinjlim K(\Hi_n) \to K (\Hi_\infty )$ is an isomorphism, and
\begin{equation} \label{compacts-direct-lim-eq}
\varinjlim K(\Hi_n) \cong K (\Hi_\infty ).
\end{equation}

Next we let 
\begin{align*}
H := \{ v \in E^0 : p_v = \sum_{i=1}^k &s_{\alpha_i} s_{\beta_i}^* \text{ for some } \alpha_1, \ldots, \alpha_k, \beta_1, \ldots, \beta_k \in \bigcup_{n=1}^\infty E^*(v_n) \\
&\text{ satisfying } s(\alpha_i) = s(\beta_i) = v \text{ for all } 1 \leq i \leq k \}.
\end{align*}


\noindent We shall show that $H$ is a saturated and hereditary subset of $E$.  

To show that $H$ is hereditary, we first observe that for each $i \in \N$ we have $p_{v_i} = s_{e_i} s_{e_i}^*$ and that $s(e_i) = v_i$ and $r(e_i) = v_{i+1}$, implying that $v_i \in H$.  Thus $\{ v_1, v_2, \ldots \} \subseteq H$.  Next, suppose that $e \in E^1$ and $s(e) \in H$.  If $s(e) = v_i$ for some $i \in \N$, then from the previous sentence we have that $r(e) = v_{e+i} \in H$.  If $s(e) \neq v_i$ for all $i \in \N$, we use the fact that $s(e) \in H$ to write 
$$p_{s(e)} = \sum_{i=1}^k s_{\alpha_i} s_{\beta_i}^*$$
for some $\alpha_1, \ldots, \alpha_k, \beta_1, \ldots, \beta_k \in \bigcup_{n=1}^\infty E^*(v_n)$ with $s(\alpha_i) = s(\beta_i) = s(e)$ for all $1 \leq i \leq k$, and moreover, the fact that $s(e) \neq v_i$ for all $i \in \N$ implies that $\alpha_i$ and $\beta_i$ are paths of length at least 1 for each $1 \leq i \leq k$.  Consequently,
\begin{equation} \label{projection-range-in-H-eq}
p_{r(e)} = s_e^* s_e = s_e^* p_{s(e)} s_e = s_e^* \left( \sum_{i=1}^k s_{\alpha_i} s_{\beta_i}^* \right) s_e =  \sum_{i=1}^k s_e^* s_{\alpha_i} s_{\beta_i}^* s_e.
\end{equation}
For each $1 \leq i \leq k$, we may use the fact that $\alpha_i$ and $\beta_i$ have lengths at least 1 to write $\alpha_i = f_1 \ldots f_l$ and $\beta_i = g_1 \ldots g_m$ for edges $f_1, \ldots, f_l, g_1, \ldots, g_m \in E^1$, and then we have
$$
s_e^* s_{\alpha_i} s_{\beta_i}^* s_e = \begin{cases} s_{f_2 \ldots f_l} s_{g_2 \ldots g_m}^* & \text{ if $f_1 = e$ and $g_1 = e$} \\
0 & \text{ otherwise.}
\end{cases}
$$
For the nonzero case above, we see that $s(f_2) = r(e)$ and $s(g_2) = r(e)$, and also $r(f_l) = r(g_m) = r(\alpha_i) = r(\beta_i)$ so that $s_e^* s_{\alpha_i} s_{\beta_i}^* s_e = s_{f_2 \ldots f_l} s_{g_2 \ldots g_m}^*$ has the properties given in defining the set $H$.  Consequently,  \eqref{projection-range-in-H-eq} shows that $r(e) \in H$.  Hence $H$ is hereditary.

To see that $H$ is saturated, suppose that $v \in E^0$ is a regular vertex with $r(s^{-1}(v)) \subseteq H$.  For each $e \in s^{-1}(v)$, the fact that $r(e) \in H$ allows us to write  
$$p_{r(e)} = \sum_{i=1}^{k_e} s_{\alpha^e_i} s_{\beta^e_i}^*$$
for some $\alpha^e_1, \ldots, \alpha^e_k, \beta^e_1, \ldots, \beta^e_k \in \bigcup_{n=1}^\infty E^*(v_n)$ with $s(\alpha^e_i) = s(\beta^e_i) = v$ for all $1 \leq i \leq k_e$.  Hence
\begin{align*}
p_v &= \sum_{s(e) = v } s_e s_e^* = \sum_{s(e) = v } s_e p_{r(e)} s_e^* = \sum_{s(e) = v } s_e \left( \sum_{i=1}^{k_e} s_{\alpha^e_i} s_{\beta^e_i}^* \right) s_e^* \\
&=  \sum_{s(e) = v } \sum_{i=1}^{k_e} s_e s_{\alpha^e_i} s_{\beta^e_i}^*s_e^* =  \sum_{s(e) = v } \sum_{i=1}^{k_e} s_{e\alpha^e_i} s_{e\beta^e_i}^* 
\end{align*}
and since $s(e\alpha^e_i) = s(e\beta^e_i) = v$ and $r(e\alpha^e_i) = r(e\beta^e_i) \in \bigcup_{n=1}^\infty E^*(v_n)$, it follows that $v \in H$.  Thus $H$ is saturated.

Since $H$ is a nonempty saturated hereditary subset, and since $C^*(E)$ is simple by Lemma~\ref{Naimark-implies-simple-lem}, it follows from Theorem~\ref{simple-theorem} that $H = E^0$.  Consequently, for any $v \in E^0$ we have that $p_v = \sum_{i=1}^k s_{\alpha_i} s_{\beta_i}^*$ for some paths  $\alpha_1, \ldots, \alpha_k, \beta_1, \ldots, \beta_k \in \bigcup_{n=1}^\infty E^*(v_n)$ satisfying $s(\alpha_i) = s(\beta_i) = v$  for all $1 \leq i \leq k$.  Thus $p_v \in  \overline{\bigcup_{n=1}^\infty A_n}$.

Likewise, for any $e \in E^1$, we have $r(e) \in H$ and $p_{r(e)} = \sum_{i=1}^k s_{\alpha_i} s_{\beta_i}^*$ for some $\alpha_1, \ldots, \alpha_k, \beta_1, \ldots, \beta_k \in \bigcup_{n=1}^\infty E^*(v_n)$ satisfying $s(\alpha_i) = s(\beta_i) = r(e)$  for all $1 \leq i \leq k$.  Thus $s_e = s_e p_{r(e)} = \sum_{i=1}^k s_e s_{\alpha_i} s_{\beta_i}^* =  \sum_{i=1}^k s_{e \alpha_i} s_{\beta_i}^* \in  \overline{\bigcup_{n=1}^\infty A_n}$.

Hence $\{ p_v, s_e : v \in E^0, e \in E^1 \} \subseteq  \overline{\bigcup_{n=1}^\infty A_n}$, and it follows that
\begin{equation} \label{C*-is-union-eq}
C^*(E) =  \overline{\bigcup_{n=1}^\infty A_n}.
\end{equation}
Combining \eqref{compact-union-eq}, \eqref{compacts-direct-lim-eq}, and \eqref{C*-is-union-eq} gives the desired result.
\end{proof}

\begin{theorem} \label{Naimark-row-countable-thm}
If $E$ is a row-countable directed graph such that $C^*(E)$ has a unique irreducible representation up to unitary equivalence, then $C^*(E) \cong K(\Hi)$ for some Hilbert space $\Hi$.
\end{theorem}

\begin{proof}
It follows from Lemma~\ref{Naimark-implies-simple-lem} that $C^*(E)$ is simple, and since $E$ is row-countable,  Proposition~\ref{three-equivalent-Naimark-prop} implies that $C^*(E)$ is AF.  The result then follows from Theorem~\ref{Naimark-AF-thm}.
\end{proof}

\end{document}